\documentclass[12pt,oneside,a4paper,leqno]{article}
\usepackage[utf8]{inputenc}
\usepackage[english]{babel}
\usepackage[all]{xy}
\CompileMatrices

\newdir{ >}{!/10pt/@{ }*@{>}}
\newdir{< }{!/10pt/@{ }*@{<}}

\usepackage{graphicx}
\usepackage{amsfonts,amssymb}
\usepackage[centertags]{amsmath}

\usepackage{amsthm} 
\newcounter{theorems}

\theoremstyle{plain}

\swapnumbers
\newcounter{lemma}
\numberwithin{equation}{section}

\newtheoremstyle{par}
     {\topsep}
     {\topsep}
     {\itshape}
     {}
     {\bfseries}
     {}
     {.5em}
     {}

\newtheoremstyle{parrm}
     {\topsep}
     {\topsep}
     {\normalfont}
     {}
     {\itshape}
     {}
     {.5em}
     {}

\theoremstyle{plain}
\numberwithin{equation}{section}

\newtheorem{theo}[equation]{Theorem}
\newtheorem{coro}[equation]{Corollary}

\theoremstyle{definition}

\newtheorem{example}[equation]{Example}

\theoremstyle{remark}
\newtheorem{remark}[equation]{Remark}

\theoremstyle{par}

\newtheorem{propo}[equation]{}

\theoremstyle{parrm}

\newcommand{\ind}{\operatorname{ind}}

\makeatletter
\def\tagform@#1{\maketag@@@{\ignorespaces#1\unskip\@@italiccorr}}

\makeatother

\usepackage{paralist}
\usepackage{comment}
\usepackage{subfigure}
\newcommand{\RR}{\mathbb{R}}
\newcommand{\CC}{\mathbb{C}}
\newcommand{\ZZ}{\mathbb{Z}}
\newcommand{\PP}{\mathbb{P}}
\newcommand{\FF}{\mathbb{F}}
\renewcommand{\SS}{\mathbb{S}}
\newcommand{\MM}{\mathbb{M}}

\newcommand{\from}{\colon}

\usepackage{bm}
\newcommand{\simbolovettore}[1]{{\boldsymbol{#1}}}

\newcommand{\vq}{\simbolovettore{q}}

\newcommand{\vu}{\simbolovettore{u}}
\newcommand{\vv}{\simbolovettore{v}}
\newcommand{\vw}{\simbolovettore{w}}

\newcommand{\vN}{\simbolovettore{N}}

\newcommand{\zero}{\boldsymbol{0}}

\newcommand{\Fix}{\operatorname{Fix}}

\DeclareMathOperator{\Ima}{Im}

\usepackage[textwidth=28mm]{todonotes}
\presetkeys{todonotes}{size=\scriptsize,backgroundcolor=yellow!30,linecolor=red}{}

\usepackage{multirow,rotating}

\newcommand{\term}[1]{\emph{#1}}

\newcommand{\grad}{\operatorname{grad}}
\newcommand{\hess}{\operatorname{Hess}}

\newcommand{\eucnorm}[1]{\lvert{#1}\rvert}
\newcommand{\mnorm}[1]{{\lVert{#1}\rVert}_M}
\newcommand{\mscalar}[2]{{\langle{#1},{#2}\rangle}_{M}}

\begin{document}
\pagenumbering{arabic}

\title{%
Central configurations, Morse and fixed point indices
}

\author{D.L.~Ferrario}

\date{%
\today}
\maketitle

\begin{abstract}
We compute the fixed point index of 
non-degenerate central configurations for the $n$-body problem 
in the euclidean space of dimension $d$, 
relating it to the Morse index of the gravitational potential 
function $\bar U$
 induced on the manifold of all maximal $O(d)$-orbits.
In order to do so, we analyze the geometry of maximal orbit type
manifolds, and compute Morse indices with respect to the mass-metric
bilinear form  on configuration spaces.  

\noindent {\em MSC Subject Class\/}: 
\vspace{0.5truecm}
70F10, 55M20

\end{abstract}

\section{Introduction: central configurations as critical points}

Let $E=\RR^d$ be the $d$-dimensional euclidean space, for $d\geq 1$. 
Fix an integer $n\geq 2$. 
The \term{configuration space} of $n$ (colored) points in $E$ is the set of all $n$-tuples
of distinct points in $E$, and denoted by $\FF_n(E)$:
\[
\FF_n(E) = \{ \vq \in E^n : i\neq j \implies \vq_i \neq \vq_ j \} = E^n \smallsetminus \Delta, 
\] 
where if $\vq\in E^n$, its $n$ components  are denoted by 
$\vq_j$, $j=1,\ldots, n$; 
points in $\FF_n(E)$ are termed \term{configurations} of $n$ points in $E$;
its complement in $E^n$ is the set of \term{collisions} 
\[
\begin{aligned}
\Delta & = \{ \vq \in E^n : \exists (i,j), i\neq j : \vq_i = \vq_j \} \\
& = \bigcup_{1\leq i < j \leq n} \{ \vq\in E^n : \vq_i = \vq_j \}.
\end{aligned}
\]

For $j=1,\ldots, n$ let $m_j>0$ be  fixed parameters (that can be interpreted as 
the mass of the $j$-th particle in $E$), under the normalization condition
\[
\sum_{j=1}^n m_j = 1~.
\]
If $\vv,\vw$ are  vectors in (the tangent space of) $E^n$,
then let 
\[
\mscalar{\vv}{\vw} = 
\sum_{j=1}^n m_j \vv_j  \cdot \vw_j 
\]
denote the mass scalar product of $\vv$ and $\vw$,
where $\vv_j \cdot \vw_j$ is the standard euclidean 
scalar product (in $E$) of the $j$-th components 
of $\vv$ and $\vw$. 
The unit sphere in $\FF_n(E)$ is termed the \term{inertia ellipsoid} and denoted by 
\[
\SS = \SS_n(E) = \{ \vq \in \FF_n(E) : \mnorm{\vq}^2 = 1 \}~.
\]
It is equal to the unit sphere/ellipsoid in $E^n$, with collisions removed,
$\SS_n(E) = S_n(E) \smallsetminus \Delta$. 
The unit sphere/ellipsoid in $E^n$ is denoted by 
$S_n(E) = \{ \vq \in E^n : \mnorm{\vq}^2 = 1\}$. To simplify notation,
if possible we will use the short forms $\SS$ and $S$ instead of 
$\SS_n(E)$ and $S_n(E)$. 

The \term{potential function} $U\from \FF_n(E) \to \RR$ is simply defined as 
\[
\sum_{1\leq i<j \leq n} \dfrac{m_im_j}{\eucnorm{\vq_i - \vq_j}^\alpha},
\]
given a fixed parameter $\alpha>0$. For $\alpha=1$, $U$ is the 
gravitational potential. It is invariant under the full group of isometries
of $E$, acting diagonally on $\FF_n(E)$.

Let $D=\nabla$ denote the covariant derivative (which is the 
Levi-Civita connection with respect to the mass-metric) in $\FF_n(E)$,
which is again the standard derivative. 
If $F\from \FF_n(E) \to E$ is a smooth function,
then $DF = dF$ is the differential of $F$, which is a section of 
the cotangent bundle
$T^* \FF_n(E)$ defined as $DF[\vv] = D_\vv F$ for each vector field $\vv$ on 
$\FF_n(E)$. 
If $\vv$ and $\vw$ are two vector fields on $\FF_n(E)$, then 
$D_\vv \vw$ is the (Euclidean and covariant) derivative of $\vw$ in the direction of $\vv$.

Let $\nabla^S$ denote the covariant derivative (Levi-Civita connection) on $S$, 
induced by the mass-metric of $\FF_n(E)$ restricted to $S$, i.e. the 
restriction to $S$ of the Riemannian structure of $\FF_n(E)$.  
If $\vv$ and $\vw$ are two vector fields defined in a neighborhood of $S$, 
then the covariant derivative $\nabla^S_\vv \vw$ 
is equal,
at $x\in S$, 
to the orthogonal projection of $D_\vv \vw$, projected orthogonally 
to the tangent space $T_xS$
(cf. proposition 3.1 at page 11 of \cite{kobayashivol2},
or proposition 1.2 at page 371 of \cite{Lang1999}). The same holds
with $\SS\subset S$ instead of $S$. 
If $\Pi$ denote the projection 
$T \FF_n(E) \mapsto 
T \SS$, 
then $\nabla^S _\vv \vw = \Pi D_\vv \vw$. 

If $F\from \FF_n(E) \to \RR$ is a smooth function, and $f = F|_\SS$ is its restriction to $\SS$,
then $\nabla^S f = df$ is the restriction of $dF$ to the tangent bundle $T\SS$. 
Let $\grad(f)=df^\sharp$ and $\grad(F)=dF^\sharp $ 
denote the gradients of $f$ and $F$ respectively (i.e., the images of the differentials
under the musical isomorphisms induced by the mass-metric). For each $x\in \SS$,
$df^\sharp(x) \in  T_x \SS$  and
$dF^\sharp(x) \in T_x \FF_n(E)$ 
satisfy the equations
\[
\mscalar{df^\sharp}{\vv} 
= df[\vv] =
\mscalar{dF^\sharp}{\vv} = 
dF[\vv]  
\]
for any $\vv \in T_x \SS$, and hence 
$\grad(f) = df^\sharp$ is the projection of $\grad(F)=dF^\sharp$ on the tangent space $T_x \SS$.
A \term{critical point} of $f$ is a point $x\in \SS$ such that 
$df=0 \iff \grad(f) = 0$, which is equivalent to say that 
$\grad(F)$ is orthogonal to $T_x \SS$. 

The \term{Hessian} of the function $f$, at a critical point $x$ of $f$ in $\SS$,
 is (cf. page 343 of \cite{Lang1999}) equal to 
the bilinear form $\hess(f)[\vv,\vw]$, defined on the tangent space $T_x \SS$
as 
\[
\hess(f)[\vv,\vw] (x) = 
\nabla^S_\vv \nabla^S_\vw f - \nabla^S_{\nabla^S_\vv \vw} f) (x) =  
(\nabla^S_\vv \nabla^S_\vw f) (x)
\]
where $\vv$ and $\vw$ are two vector fields defined in a neighborhood of $x$. 

The Hessian of $F$ is simply the symmetric matrix of all the second derivatives
$D^2 F$:
\[\begin{aligned}
\hess(F)[\vv,\vw](x) & = (D_\vv D_\vw F) (x) = D^2F(x) [\vv,\vw]  \\
& = 
\sum_{\substack{i=1\ldots n\\ \beta=1,\ldots, d}} \sum_{\substack{j=1,\ldots, n\\\gamma=1,\ldots,d}}
\dfrac{\partial^2 F}{\partial \vq_{i\beta} \partial \vq_{j\gamma}} 
\vv_{i\beta} \vw_{j\gamma}
\end{aligned}
\] 
where $\vq_{i\beta},  \vv_{i\beta}$ and $\vw_{j\gamma}$ 
are the $d$ cartesian components in $E$  ($\RR^d$ as the tangent space of $E$) 
of $\vq_i$, $\vv_i$ and $\vw_j$ respectively. 

Using the mass-metric, 
 if $\vN$ denotes the 
unit vector field normal to $T_x \SS$ in $T_x \FF_n(E)$, 
the projection of $\nabla^S_\vv \vu$ of any vector field $\vu$ on $T_x \SS$ is 
\[
\nabla^S_\vv \vu = D_\vv \vu - \mscalar{ D_\vv \vu}{\vN} \vN ,
\]
and
\[
df^\sharp  = dF^\sharp  - \mscalar{dF^\sharp}{\vN} \vN~.
\]
The Hessian can be written also as
(cf. page 344 of \cite{Lang1999}) 
$ \hess(f)[\vv,\vw] (x)   = 
 \mscalar{ \nabla^S_\vv df^\sharp}{\vw} $
and 
$\hess(F)[\vv,\vw] (x) = \mscalar {D_\vv dF^\sharp}{\vw}$. 
It follows therefore that 
\[\begin{aligned}
\hess(f)[\vv,\vw] (x)  & = 
 \mscalar{ \nabla^S_\vv df^\sharp}{\vw} \\
& = 
\mscalar{ \nabla^S_\vv 
\left( dF^\sharp  - \mscalar{dF^\sharp}{\vN} \vN  \right)
}{\vw} 
\\
& = 
\mscalar{ \nabla^S_\vv \left(
 dF^\sharp  \right)
}{\vw} 
-
\mscalar{ \nabla^S_\vv\left( 
\mscalar{dF^\sharp}{\vN} \vN \right)  }{\vw}.
\\
\end{aligned}
\]
Because of the product rule for each function $\varphi$ and each 
vector field $\vu$
\[
\nabla^S_\vv\left( \varphi \vu \right)  = 
\varphi \nabla^S_\vv \vu + (d\varphi[\vv]) \vu 
\]
\[
\implies 
\nabla^S_\vv\left( 
\mscalar{dF^\sharp}{\vN} \vN
\right)  = 
\mscalar{dF^\sharp}{\vN} \nabla^S_\vv \vN + 
d\left( \mscalar{dF^\sharp}{\vN} \right)[\vv] \vN 
\]
which implies that 
\[
\mscalar{ \nabla^S_\vv\left( 
\mscalar{dF^\sharp}{\vN} \vN \right)  }{\vw} 
= 
\mscalar{dF^\sharp}{\vN} 
\mscalar{  
\nabla^S_\vv \vN}{\vw} 
\]
since $\vN$ is orthogonal to $\vw$. 
The same argument can be applied to show that
for any vector field $\vu$ (not necessarily tangent to $\SS$)
\[
\mscalar{\nabla^S_\vv \vu}{\vw}  = \mscalar{D_\vv \vu}{\vw}, 
\]
and therefore that, evaluated at the critical point $x$, 
\[\begin{aligned}
\hess(f)[\vv,\vw] & =  
\mscalar{ D_\vv \left(
 dF^\sharp  \right)
}{\vw} 
-
\mscalar{dF^\sharp}{\vN} 
\mscalar{  
D_\vv \vN}{\vw} 
\\
& = D^2 F [\vv, \vw] 
-
\mscalar{dF^\sharp}{\vN} 
\mscalar{  
D_\vv \vN}{\vw} 
\end{aligned}
\]

The inertia ellipsoid $S$ is defined by the equation $\mnorm{\vq}^2 = 1$,
or equivalently $h(\vq) = \frac{1}{2}$ where $h(\vq) = \dfrac{1}{2} \mnorm{\vq}^2$. 
The normal unit vector $\vN$ is equal to $dh^\sharp = \vq$, 
and thus 
\[\begin{aligned}
\hess(f)[\vv,\vw] & =  
 D^2 F [\vv, \vw] 
-
\mscalar{dF^\sharp}{\vq} 
\mscalar{  
D_\vv \vq}{\vw} 
\\
& = D^2 F [\vv, \vw] 
-
\mscalar{dF^\sharp}{\vq} 
\mscalar{\vv}{\vw} 
\\
\end{aligned}
\]

If $F=U$, then $U$ is homogeneous of degree $-\alpha$, and therefore
$\mscalar{dU^\sharp}{\vq} =  dU(\vq)[\vq] = -\alpha U(\vq)$.
The following equation follows, at any critical point $x$ of the restriction 
of $U$ to $\SS$.
\begin{equation}
\label{eq:hess1}
\hess(U|_\SS)[\vv,\vw]  =  
 D^2 U(x) [\vv, \vw] 
+
\alpha U(x) 
\mscalar{\vv}{\vw} 
\end{equation}

A \term{central configuration} is a configuration $\vq\in \FF_n(E)$ with the property that 
there exists
a multiplier 
$\lambda\in \RR$ 
such that 
\begin{equation}
\label{eq:CC}
dU^\sharp(\vq) = \lambda \vq,
\end{equation}
where $dU^\sharp$ is the gradient in $E^n$ of the 
potential function $U$, with respect to the mass-metric. Equation 
\eqref{eq:CC} implies that $ \lambda = -\alpha \dfrac{U(\vq)}{\mnorm{\vq}^2}$
(for more on central configurations see e.g. \cite{wintner} (§369--§382bis at pp.~284--306),
\cite{MR2139425},
\cite{Mo90},
\cite{MR856309},
\cite{Xia},
\cite{MR1320359},
\cite{HaM2207019},
\cite{MR2925390},
\cite{MR3469182},
\cite{ferrario2016central}).
An equivalent definition for a normalized (i.e. $\vq\in \SS$) central configuration is the 
following:
\begin{propo}
$\vq\in \SS_n(E)$ is a central configuration if and only if it is a critical point
for the restriction $U|_{\SS}$ of the potential function to $\SS=\SS_n(E)$.
\end{propo}

Let $c\from E^n \to E^n$ be the isometry defined as 
$c(\vq) = \vq'$, with 
\begin{equation}
\label{eq:defc}
\vq'_j = \vq_j - 2\vq_0
\end{equation}
 for each $j=1,\ldots,n$, 
and with $\vq_0 = \sum_{j=1}^n m_j \vq_j$. It is an isometry, since
$\mnorm{\vq'}^2 = \sum_{j=1}^n m_j \eucnorm{\vq_j - 2\vq_0}^2 = 
\sum_{j=1}^n m_j(\eucnorm{\vq_j}^2 + 4 \eucnorm{\vq_0}^2  - 4 \vq_j \cdot \vq_0) =
\sum_{j=1}^n m_j \eucnorm{\vq_j}^2 + 4 (\sum_{j=1}^n m_j) \eucnorm{\vq_0}^2 -  4 \eucnorm{\vq_0}^2
= \mnorm{\vq}^2$. It is the orthogonal reflection around the space  
of all configurations with center of mass $\vq_0$ equal to zero:
$c\vq = \vq \iff \vq_0 =0$. It is easy to see that if $\vq$ is a central configuration
then $c\vq = \vq$, and hence $\vq$ has center of mass $\vq_0$ in $0$. 
Let $Y$ be defined as $Y = \{ \vq \in E^n : \vq_0 = \zero \}$,
and $\SS^c = \SS \cap Y$, $S^c = S \cap Y$. In other words,
elements of $\SS^c$ are normalized configurations with center of mass in $0$. 
Since the potential function is invariant up to translations, 
$U(c\vq) = U(\vq)$, and any critical point of the restriction $U|_{\SS^c}$ 
is a critical point of $U|_{\SS}$ (for example, by Palais Principle of Symmetric Criticality 
\cite{palais}).
Thus it is equivalent to define 
central configurations as critical points of $U|_{\SS^c}$
or as critical points of $U|_{\SS}$.

\section{Fixed points, $SO(d)$-orbits and projective configuration spaces}

Following 
\cite{MR2372989,Fer2015},
consider the function $F\from \SS_n(E) \to S_n(E)$ defined as 
\begin{equation}
\label{eq:F}
F (\vq) = - \dfrac{ dU^\sharp (\vq) }{\mnorm{dU^\sharp(\vq)} } 
\end{equation}
where $dU^\sharp$ is the gradient of $U$, with respect to the mass-metric.

First, consider the isometry $c$ defined above in \eqref{eq:defc}. 
Since $F(c\vq) = c F(\vq)$, $F(\SS^c) \subset S^c$. Moreover, 
as the image of $F$ is in $S^c$, if $F^c$ denotes the restriction $F^c \from \SS^c \to S^c$,
\begin{equation}
\label{eq:isom1}
\Fix(F^c) = \Fix(F),
\end{equation}
and the fixed point indexes are exactly the same. 

Let $O(d)$ be the special orthogonal group, acting diagonally on $E^n$,
and $SO(d)$ the special orthogonal subgroup. The inertia ellipsoid
$\SS$, $S$ and $Y$ are $O(d)$-invariant in $E^n$, 
and so are $\SS^c$ and $S^c$. 
Let $\pi \from S \to S/G$ denote the quotient map onto the space of $G$-orbits,
for $G=SO(d)$ or $G=-O(d)$.  

Since $U$ is a $G$-invariant function, 
$F$ is a $G$-equivariant map, and hence it induces a map on the quotient spaces:
\begin{equation}\label{quotientspaces}
\xymatrix{
\SS \ar[r]^F \ar[d]_\pi  & S \ar[d]^\pi \\
\SS/G \ar[r]_{f} & S/G
}
\end{equation}

A fixed point of $F$ is a normalized configuration $\vq$ such that 
$F(\vq)=\vq$. 
A fixed point of $f$ is a conjugacy class $[\vq]$ of configurations
such that $f([\vq])=[\vq]$, i.e. it is a conjugacy class $[\vq]$ 
such that $F(\vq) = g \vq$ for some $g\in G$.  
It follows from Theorem (2.5) of \cite{Fer2015}  that if $G=SO(d)$, 
then 
$F(\vq) = g\vq$ $\iff$ $F(\vq) =\vq$, or equivalently that 
\begin{equation}
\label{eq:fer2015}
G=SO(d) \implies 
\pi(\Fix(F)) = \Fix(f),
\end{equation}
and hence also that 
$\pi(\Fix(F^c)) = \Fix(f^c)$. 

\begin{remark} 
Elements in $\SS/G$ are called projective configurations: for $d=2$ and $G=SO(2)$, 
$S / G$ is the $(n-1)$-dimensional complex projective space 
$\PP^{n-1}(\CC)$, and $S^c$ is a hyperplane in it,
hence a $(n-2)$-dimensional complex projective space 
$\PP^{n-2}(\CC)$
For $n=3$, it is the Riemann sphere. Projective configurations are 
projective classes of elements $[\vq_1:\vq_2:\vq_3]$ in $\PP^1(\CC)\subset \PP^2(\CC)$ 
such that $m_1 \vq_1 + m_2 \vq_2 + m_3 \vq_3 =0$, 
$\vq_j \in \CC$, and $\vq_1 \neq \vq_2$, $\vq_1 \neq \vq_3$, $\vq_2 \neq \vq_3$.  

For $d=1$, projective configurations are equivalence classes under the action of the orthogonal 
group $G=O(1) = \ZZ_2$. 
\end{remark}

The following Corollary of \eqref{eq:fer2015}
 shows that the difference is minor. 

\begin{coro}
\label{coro:fer2015}
If $\vq \in \SS$ is a central configuration such that 
$F(\vq) = g \vq$, with $g\in O(d)$ (acting diagonally on $E^n$), 
then $g=1$. 
\end{coro}
\begin{proof}
Let $E' = E \oplus \RR$ be the euclidean space of dimension $d+1$,
and $E\subset E'$ one of its $d$-dimensional subspaces. 
If $\vq\in \SS \subset \FF_n(E)$, then 
$\vq\in \SS \subset \FF_n(E) \subset \FF_n(E')$,
and there exists $g' \in SO(d+1)$ such that $g'E=E$ and 
the restriction of $g'$ to $E$ is equal to $g$: it follows that 
$F(\vq) = g' \vq$, in $\FF_n(E')$, and therefore $g'=1$, from which it follows that 
$g=1$. 
\end{proof}

%
%

Homological calculations on configurations spaces for the sake
of central configurations have been done by 
Palmore \cite{MR0321389},
Pacella \cite{MR856309} and
McCord \cite{MR1417773}. 
We can arrange all the spaces inertia ellipsoids and the corresponding 
projective quotients as in diagram \eqref{eq:diagram}. 
\begin{equation}
\label{eq:diagram}
\begin{aligned}
\xymatrix{
\SS_n^c(\RR) \ar^{\iota_1}[r] \ar[d] & \SS_n^c(\RR^2) \ar^{\iota_2}[r] \ar[d] & 
\SS_n^c(\RR^3) \ar^{\iota_3}[r] \ar[d] & \cdots \\
\SS_n^c(\RR)/SO(1) \ar[r]^{\bar\iota_1} \ar[d] & \SS_n^c(\RR^2)/SO(2) \ar^{\bar\iota_2}[r] \ar[d] 
& \SS_n^c(\RR^3)/SO(3) \ar^{\bar\iota_3}[r] \ar[d] & \cdots \\
\SS_n^c(\RR)/O(1) \ar^{\bar\bar\iota_1}[r] & \SS_n^c(\RR^2)/O(2) \ar^{\bar\bar\iota_2}[r] 
& \SS_n^c(\RR^3)/O(3) \ar^{\bar\bar\iota_3}[r]  & \cdots 
}
\end{aligned}
\end{equation}

For each $d$, 
$\SS_n^c(\RR^d)$ is a deformation retract of 
$\FF_n^c(\RR^d)$, which in turn is a deformation retraction 
of $\FF_n(\RR^d)$ (where $\FF_n^c(E)$ denotes the space
of all configurations with center of mass in $0$). 
The Poincaré polynomial for the cohomology of the configuration space 
$\FF_n(\RR^d)$ is equal to 
\[
P(t)= \prod_{k=1}^{n-1}(1+kt^{d-1}),
\]
as shown e.g. in Theorem 3.2 of \cite{MR1634470} (see also Proposition 2.11.2 of \cite{MR3469182}).

Now, note that in the sequence of projections
\[
\SS_n^c(\RR^d) \to
\SS_n^c(\RR^d)/SO(d) \to
\SS_n^c(\RR^d)/O(d)
\]
the second map corresponds to the projection given by 
the action of the quotient group $\ZZ_2 = O(d)/SO(d)$ 
on the quotient space $\SS_n^c(\RR^d)/SO(d)$ ($SO(d)$ is normal in $O(d)$). For $d\geq 2$, 
let $h$ be the orthogonal reflection of $\RR^d$ around $\RR^{d-1}\subset \RR^d$:
its coset $hSO(d)$ is the generator of $O(d)/SO(d)$, and hence 
the image $\Ima(\bar\iota_{d-1})$
in $\SS_n^c(\RR^d)/SO(d)$ is fixed by $O(d)/SO(d)$. Actually,
it is equal to the fixed point subset of $O(d)/SO(d)$ in 
$\SS_n^c(\RR^d)/SO(d)$. Outside the image of $\bar\iota_{d-1}$, 
therefore the $\ZZ_2$ action is free: let $\MM_n(\RR^d)$ denote the 
manifold
\begin{equation}\label{def:M}
\MM_n(\RR^d) 
= 
\left( \SS_n^c(\RR^d)/SO(d) \smallsetminus \Ima(\bar\iota_{d-1}) \right) / \ZZ_2 
= \SS_n^c(\RR^d)/O(d) \smallsetminus \Ima(\bar\bar\iota_{d-1}),
\end{equation}
where the last equality holds since $\bar\iota_{d-1}$ factors through
$\SS_n^c(\RR^{d-1})$.

The next proposition follows from the dimension of $SO(d)$ and the
previous remarks.  

\begin{propo}
The subspace of all points in $\SS_n^c(\RR^d)/O(d)$
with maximal orbit type is the open subspace
$\MM_n(\RR^d)$
defined in \eqref{def:M}, and it is
is a manifold of dimension 
\[
\dim \MM_n(\RR^d) = 
d(n-1)-1 - d(d-1)/2  
\]
\end{propo}


For $d=1$, it is the projective space $\PP^{n-2}(\RR)$ minus collisions. 
For $d=2$, it is a $(2n-4)$ dimensional manifold (where 
$\PP^{n-2}(\CC)$ minus collinear 
and minus collisions is its double cover). 

\begin{propo}
$\SS^c_n(\RR^2)/SO(2)$ has the same homotopy type of $\FF_{n-2}(\RR^2\smallsetminus \{p,q\})$,
where $p,q$ are two arbitrary distinct points of $\RR^2$. 
\end{propo}
\begin{proof}
It is Lemma 4.1 of \cite{MR1417773}.
\end{proof}

It follows that the Poincaré polynomial (where $\beta_j$ are 
Betti numbers) of $\SS^c_n(\RR^2)/SO(2)$ is
\begin{equation}
\label{eq:pp}
p(t) = 
\prod_{k=2}^{n-1} (1+kt) 
= \sum_{j=0}^{n-2} \beta_j t^j.
\end{equation}
(see also Proposition 2.11.3 of \cite{MR3469182} ).
McCord in \cite{MR1417773} proved also that 
\[
\dim H^k (\MM_n(\RR^2)) = \begin{cases}
\sum_{j=0}^k \beta j & \text{if $k\leq n-3$} \\
0 & \text{otherwise},
\end{cases}
\]
while Pacella in (2.4) of 
\cite{MR856309}
computed the $SO(3)$-equivariant homology (using Borel homology) Poincaré series of 
$\SS^c_n(\RR^3) \sim \FF_n(\RR^3)$ as 
\[
P^{SO(3)}(t) = 
\dfrac{ \prod_{k=2}^{n-1}(1+kt^{2}) }
{1-t^2}
\]

\begin{remark}
The projective quotient 
$\SS^c_n(\RR^2)/SO(2)$ 
is a manifold (it is the projective space 
$\PP^{n-2}(\CC)$ with collisions removed). 
It contains $\SS^c_n(\RR)/O(1)$ 
as a submanifold (the collinear configurations). 
For $d\geq 3$ the isotropy groups of the action start being non-trivial,
and the filtration of subspaces of constant 
orbits type in  $\SS^c_n(\RR^d)/SO(d)$ 
is given by the horizontal arrows $\bar\iota_j$ in diagram
\eqref{eq:diagram}.
\end{remark}

\section{Fixed points and Morse indices}

Let $\vq\in \SS_n^c(\RR^d)$ a central configuration, and hence a fixed point of 
the map $F$  defined 
above in \eqref{eq:F}, 
such that its $O(d)$-orbits lies in the maximal orbit type 
submanifold $\MM_n(\RR^d) \subset \SS^c(\RR^d)/O(d)$.

\begin{propo}
\label{propo:EC}
If $DF\from T_\vq \SS \to T_\vq \SS$  denotes the differential of 
$F$ at the central configuration $\vq$, then for any $\vv,\vw \in T_\vq \SS$ the following
equation holds:
\[
 D^2 U(\vq) [\vv,\vw]  = -
\alpha U(\vq)\mscalar{DF[\vv]}{\vw} 
\]
\end{propo}
\begin{proof}
As we have seen in the introduction, 
$\mscalar{ D_\vv dU^\sharp}{ \vw} = D^2U [\vv,\vw]$,
and if $\vq$ is a normalized central configuration then 
by \eqref{eq:CC} $dU^\sharp (\vq)= \lambda \vq$ 
with $\lambda =   -\alpha \dfrac{U(\vq)}{\mnorm{\vq}^2}= -\alpha U(\vq)$.
It follows that  $\mscalar{dU^\sharp}{\vw} = 0$, being $\vw$ tangent to $\SS$, and
$\mnorm{dU^\sharp} = -\lambda = \alpha U(\vq)$. Also,
\[\begin{aligned}
\mscalar{DF[\vv]}{\vw}  & = 
\mscalar { D_\vv \left( 
- \dfrac{ dU^\sharp}{\mnorm{dU^\sharp}}
\right)}{\vw} 
\\
&= 
- \mscalar { \left( 
 \dfrac{ D_\vv dU^\sharp}{\mnorm{dU^\sharp}}
\right)}{\vw} 
-
\mscalar { 
  D_\vv \left( \dfrac{1}{\mnorm{dU^\sharp}}
\right) dU^\sharp }{\vw}  \\
&= 
- \dfrac{1}{\mnorm{dU^\sharp}}
 \mscalar{ D_\vv dU^\sharp}{\vw} 
- 0 
\\
&= - \dfrac{1}{\alpha U(\vq)} D^2U(\vq)[\vv,\vw].
\end{aligned}
\]
\end{proof}

Combining \ref{propo:EC} with equation 
\eqref{eq:hess1} the following corollary holds.
\begin{coro}
\label{coro:plomboplata}
If $\vq$ is as above, then 
for each $\vv,\vw \in T_\vq \SS$
\[
\hess(U|_\SS)[\vv,\vw] = \alpha U(\vq) \left(
\mscalar{\vv}{\vw} - \mscalar{DF[\vv]}{\vw} 
\right)
\]
\end{coro}

Finally, 
consider again the group $O(d)$ acting on $\SS_n^c(\RR^d)$. Let $F$ and $\vq$ be 
the map and the central configuration defined above.  
Recall that $f\from \SS/O(d) \to S/O(d)$ denotes the map defined on the quotient. 
Let $[\vq] \in \MM_n(\RR^d) \SS/O(d)$ 
denote the projective class (i.e. the $O(d)$-orbit of $\vq$) of $\vq$,
which is a fixed point of $f$, and is a critical point of the map 
$\bar U \from \MM_n(\RR^d) \to \RR$
induced on $\MM_n$ by $U$, defined simply as 
$\bar U([x]) = U(x)$ for each $x\in \SS_n^c(\RR^d)$. 

\begin{theo}
\label{theo}
The point $[\vq]$ is a non-degenerate critical point of $\bar U$ if and only if 
it is a non-degenerate fixed point of $f$. If $\ind([\vq],f)$ 
denotes the fixed point index of $[\vq]$ for $f$, and $\mu([\vq])$ the Morse
index of $[\vq]$, then the following equation holds:
\[
\ind([\vq],f) = (-1)^ { \mu([\vq])}.
\]
\end{theo}
\begin{proof}
The point $[\vq]$ is a non-degenerate  critical point if and only if 
the dimension of the kernel of the Hessian $\hess(U|_\SS)(\vq)$ 
is equal to the dimension of $SO(d)$, i.e. $d(d-1)/2$. 
By \ref{coro:plomboplata}, 
the kernel is equal to the
eigenspace of $DF(\vq)$ corresponding to the eigenvalue $1$, 
which has dimension $d(d-1)/2$ if and only if the fixed point $[\vq]$ 
is non-degenerate. 
Now, if this holds then the index $\ind([\vq],f)$ is equal to the number $(-1)^e$.
where $e$ is the number of negative eigenvalues  $1 - f'$, 
which is the same as the number of negative eigenvalues of $1-F'$.
Again by \ref{coro:plomboplata} and since $U>0$,
 $e$ is equal to the number of negative eigenvalues
of 
$\hess(U|_\SS)$, which is by definition the Morse index $\mu([\vq])$.
\end{proof}

\begin{remark}
Unfortunately, a former version of this 
statement had a wrong formula for $\ind(\vq)$. 
In fact, in (3.5) of \cite{Fer2015} one should put $\epsilon=0$, and not 
$\epsilon = d(n-1)-1-d(d-1)/2  = \dim \MM_n(\RR^d)$.
The error occurred because I used the wrong sign of $U$ in (3.1) ($V=-U$ instead of $U$). 
\end{remark}

\begin{example}
For $d=1$ and any $n$, all critical points are local minima of $U$,
and hence $\mu=0$, and fixed points have index $1$. 
The map induced on the quotient can be regularized on binary collisions (see 
\cite{Fer2015,MR2372989}), hence the map on the quotient can be extended to a self-map
$f\from \PP^1(\RR) \to \PP^1(\RR)$ with three fixed points of index $1$.
Therefore the Lefschetz number of $f$ is $3$, and $f$ has degree $-2$.

For $d=2$ and $n=3$, 
the three Euler configurations have  $\mu=1$, while 
the two Lagrange points have $\mu=1$, hence the map $f$ 
induced on the quotient $\PP^1(\CC)$ (again,
by regularizing the binary collisions) has Lefschetz number  equal to
$L(f) = 2 -3 = -1$. Therefore the degree of $f$ is equal to $-2$. 
\end{example}




\end{document}